\documentclass[12pt,reqno]{amsart}
\usepackage{amssymb, amsfonts, amsbsy, latexsym, epsfig, color}

\newtheorem{Thm}{Theorem}[section]

\newtheorem{corollary}[Thm]{Corollary}

\newtheorem{theorem}[Thm]{Theorem}

\def\ldots{\mathinner{\ldotp\ldotp\ldotp}}
\def\ldots{\mathinner{\cdotp\cdotp\cdotp}}

\def\H{\mathbb{H}}

\begin{document}

\title[Spielman and Srivastava]
{The simplified version of the Spielman and Srivastava algorithm for proving the Bourgain-Tzafriri 
restricted invertiblity theorem}
\author[P.G. Casazza 
 ]{Peter G. Casazza }
\address{Department of Mathematics, University
of Missouri, Columbia, MO 65211-4100}

\thanks{The author was supported by  NSF DMS 1008183; and  NSF ATD 1042701; 
AFOSR  DGE51:  FA9550-11-1-0245}

\email{casazzap@missouri.edu}

\maketitle

\begin{abstract}
By giving up the best constants, we will see that the original argument of Spielman and
Srivastava for proving the Bourgain-Tzafriri Restricted 
Invertibility Theorem \cite{SS} still works - and is much
simplier than the final version.  We do not intend on publishing this since it is their
argument with just a trivial modification, but we want to make it available
to the mathematics community since several people have requested it already.
\end{abstract}

\section{Introduction}

Recently, Spielman and Sristave \cite{SS} made a stunning achievement by showing that
one of the deeper and most useful results in pure mathematics, the Bourgain-Tzafriri
Restricted Invertibility Theorem \cite{BT}, can be proved directly with an algorithm.
The original proof had a technical error which they corrected in a later version.  
But this correction
{\it doubled} the degree of difficulty of the proof.  We will see that their original proof
is still valid if we are willing to give up the best constant in the theorem.

\section{The Theorem and Their Original  Proof Adjusted}

\begin{theorem}[Spielman and Srivastave]
Let $\H$ be a Hilbert space with orthonormal basis $\{v_i\}_{i=1}^n$. 
Assume $L:\H \rightarrow \H$ is a linear operator with $\|Lv_i\|=1$
for all $i=1,2,\ldots,n$ and assume 
\[ A = \sum_{i=1}^m Lv_i Lv_i^T ,\]
has $m$ non-zero eigenvalues, all of which are greater than b, and $b'=b-\delta>\delta$.
If
\[ Tr[L^T(A-bI)^{-1}L]\le -n - \frac{2\|L\|^2}{\delta},\]
then there exists a vector $\omega \in \{Lv_i\}_{i=1}^n$ satisfying:
\vskip12pt
1.   $\omega^T(A-b'I)^{-1}\omega < -1$, and hence $\omega = Lv_j$ for some $m<j\le n$.
\vskip12pt
2.   $Tr[L^T(A+\omega \omega^T-b'I)^{-1}L] \le Tr[L^T(A-bI)^{-1}L]\le 
-n -\frac{2\|L\|^2}{\delta}.$
\vskip12pt
(Note that we added a 2 to the original constant in \cite{SS} (part (2) above)
and as a result we have
to change their starting point barrier from $(1-\epsilon )$ to $(1-2\epsilon$).)
\end{theorem}

\begin{proof}

\noindent {\bf Step I:}  We show:
\[ (A-bI)^{-1} -(A-b'I)^{-1} \ge \frac{\delta}{2}  (A-b'I)^{-2},\]
\vskip12pt
\noindent {\bf Note:}  In the original paper the above inequality 
was stated to hold for $\delta$ instead of $\delta/2$.  
But this isn't
true and is not even true for real numbers.    Our fix will change their
{\it perfect constant} for the lower Riesz bound from their $(1-\epsilon)^2$
to $(1-2\epsilon)(1-\epsilon)$.  
\vskip12pt

\noindent {\bf Proof}:  Note first that $\delta \le b'$ implies
$2b' \ge b'+\delta$.  Thus 
\[ \frac{1}{b'+\delta}\ge \frac{1}{2b'},\]
and finally
\[ \frac{1}{b'(b'+\delta)} \ge \frac{1}{2(b')^2}.\]
Now,
\begin{equation}\label{E20}
 \frac{-1}{b}-\frac{-1}{b'} = \frac{b-b'}{bb'}= \frac{\delta}{b'(b'+\delta)}\ge
 \frac{\delta}{2(b')^2} .
 \end{equation}
 Also,
 \[ \lambda_i-b \le \lambda_i-b',
\]
and so
\[ (\lambda_i-b)(\lambda_i-b') \le (\lambda_i-b')^2.\]
Hence,
\[ \frac{1}{(\lambda_i-b)(\lambda_i-b')} \ge \frac{1}{(\lambda_i-b')^2},\]=
  and thus
 \begin{eqnarray*}
  \frac{1}{\lambda_i-b}-\frac{1}{\lambda_i-b'} &=& \frac{b-b'}{(\lambda_i-b)(\lambda_i-b')}\\
  &\ge& \frac{\delta}{(\lambda_i-b')^2}.
  \end{eqnarray*}

\vskip12pt
\noindent {\bf Step 2:}  We observe that
\[ Tr[L^T(A-b'I)^{-1}L] \le Tr[L^T(A-bI)^{-1}L].\]
\vskip12pt

\noindent {\bf Proof}: By Step I,  we have
\begin{eqnarray*}
 Tr[L^T(A-bI)^{-1}L - L^T(A-b'I)^{-1}L] &=& Tr[L^T((A-bI)^{-1}-(A-b'I)^{-1})L]\\
 &\ge& \frac{\delta}{2} Tr[L^T(A-b'I)^{-2}L] \ge0.
 \end{eqnarray*}
 
 \vskip12pt
 
 \noindent {\bf Step 3:}  We show
 \[ Tr[L^T(A-b'I)^{-1}LL^T(A-b'I)^{-1}L]\]
 \[ \le (Tr[L^T(A-bI)^{-1}L]-Tr[L^T(A-b'I)^{-1}L])(-n-Tr[L^T(A-b'I)^{-1}L]).\]
 \vskip12pt
 \noindent {\bf Proof:}  Since 
 \[ Tr[L^T(A-b'I)^{-1}L] \le -n-\frac{2\|L\|^2}{\delta},\]
 we have
 \[ \|L\|^2 \le \frac{\delta}{2} (-n-Tr[L^T(A-b'I)^{-1}L)]),\]
 and hence,
 \[ \|L\|^2 Tr[L^T(A-b'I)^{-2}L] \le\frac{ \delta}{2} Tr[L^T(A-b'I)^{-2}L] (-n-Tr[L^T(A-b'I)^{-1}L]).\]
Applying the proof of Step I, and the facts:  $(A-b'I)^{-1}LL^T(A-b'I)^{-1} \ge 0$
and $LL^T \le \|L\|^2I$, we have
\begin{eqnarray}\label{E2}
Tr[L^T(A-b'I)^{-1}LL^T(A-b'I)^{-1}L]
\end{eqnarray}
\[\le \|L\|^2 
  Tr[L^T(A-b'I)^{-2}L] \]
  \[ \le \frac{\delta}{2} Tr[L^T(A-b'I)^{-2}L](-n-Tr[L^T(A-b'I)^{-1}L])\]
 \[ \le (Tr[L^T(A-bI)^{-1}L] - Tr[L^T(A-b'I)^{-1}L)(-n-Tr[L^T(A-b'I)^{-1}L]).\] 
 \vskip12pt

 \noindent {\bf Step 4:}  We pick a vector $\omega$ satisfying (1) and
 \[ Tr[L^T(A-b'I)^{-1}L] -\frac{\omega (A-b'I)^{-1}LL^T(A-b'I)^{-1}\omega}{1+
 \omega (A-b'I)^{-1}\omega} \le Tr[L^T(A-bI)^{-1}L].\]
 \vskip12pt
\noindent {\bf Proof}:  Noting that $\omega^T\omega =1$, it follows from inequality \ref{E2}
 that there is a vector $\omega \in \{Lv_i\}_{i=1}^n$  so that
 \begin{eqnarray}\label{E1} 
 \omega^T(A-b'I)^{-1}LL^T(A-b'I)^{-1}\omega
 \end{eqnarray}
\[ \le (Tr[L^T(A-bI)^{-1}L] -Tr [L^T(A-b'I)^{-1}L)](-1-\omega^T(A-b'I)^{-1}\omega)
  \]
     Since the left-hand side of Equation \ref{E1} is non-negative, applying Step 1
 we have
 \[ 0 < -1- \omega^T(A-b'I)^{-1}\omega,\]
 and hence
 \[ \omega^T(A-b'I)^{-1}\omega < -1.\]
 For $1\le j \le m$, if $\omega = Lv_j$, then
 \[  \omega^T(A-b'I)^{-1}\omega =  \sum_{i=1}^m \frac{1}{\lambda_i-b'}\omega_i^2 \ge 0.\]
 So $\omega = Lv_j$ for $m<j\le n$.
  Now, Equation \ref{E1} implies
  \[ 
  \frac{\omega^T(A-b'I)^{-1}LL^T(A-b'I)^{-1}\omega}{-1-\omega^T(A-b'I)^{-1}\omega}
  \le Tr[L^T(A-bI)^{-1}L] -Tr [L^T(A-b'I)^{-1}L],\]
  and the result follows.
 \vskip12pt

 \noindent {\bf Step; 5:}   We check part (2) of the theorem.
 \vskip12pt
  \noindent {\bf Proof}:
We apply the Sherman-Morrison formula - which states, for a matrix $A$,
\[ (A+\omega \omega^T)^{-1} = A^{-1} - \frac{A^{-1}\omega\omega^TA^{-1}}{1+
\omega^TA^{-1}\omega}.\]
It follows that (Replacing $A$ by $A-b'I$)
\[ L^T(A+\omega \omega^T-b'I)^{-1}L = 
L^T(A-b'I)^{-1}L - \frac{L^T(A-b'I)^{-1}\omega\omega^T(A-b'I)^{-1}L}{1+
\omega^T(A-b'I)^{-1}\omega}.\]
Thus,
\[Tr[L^T(A+\omega \omega^T-b'I)^{-1}L] =\]
\[
Tr[L^T(A-b'I)^{-1}L] - \frac{Tr[L^T(A-b'I)^{-1}\omega\omega^T(A-b'I)^{-1}L]}{1+
\omega^T(A-b'I)^{-1}\omega}.\]
Using the fact that $Tr[AB] = Tr[BA]$, we have that the above equals
\[ Tr[L^T(A-b'I)^{-1}L]- \frac{Tr[\omega^T(A-b'I)^{-1}LL^T(A-b'I)^{-1}\omega]}
{1+\omega^T(A-b'I)^{-1}\omega}=\]
\[ Tr[L^T(A-b'I)^{-1}L]- \frac{\omega^T(A-b'I)^{-1}LL^T(A-b'I)^{-1}\omega}
{1+\omega^T(A-b'I)^{-1}\omega}.\]
We now have applying Step 4:

  \[Tr[L^T(A+\omega \omega^T - b'I)^{-1}L]  = Tr[(L^T(A-b'I)^{-1}L] - 
\frac{\omega^T(A-b'I)^{-1}LL^T(A-b'I)^{-1}\omega}{1+\omega^T(A-b'I)^{-1}\omega}\]
\[ \le Tr[L^T(A-bI)^{-1}L]\]
This completes the proof of the theorem.
\end{proof}

\begin {corollary}[Bourgain-Tzafriri Restricted Invertibility Theorem]
If we iterate the algorithm $k$ times, we get $k$ vectors from $\{Lv_i\}_{i=1}^m$
with lower Riesz bound for the operator $A$
\[ 1-2\epsilon - (k-1)\delta = (1-2\epsilon)(1-(k-1)\frac{\|L\|^2}{\epsilon n})\]
Hence,

1.  If
\[ k= \lceil \frac{\epsilon^2 n}{\|L\|^2}\rceil,\]
then
\begin{eqnarray*} (1-2\epsilon)\left [ 1-(k-1)\frac{\|L\|^2}{\epsilon n}\right ]&\ge&
(1-2\epsilon)\left [1-\frac{\epsilon^2n}{\|L\|^2}\frac{\|L\|^2}{\epsilon n}\right ]\\
&=& (1-2\epsilon) (1-\epsilon).
  \end{eqnarray*}
which is BT.

2.  If
\[   k= \lceil \frac{\epsilon n}{\|L\|^2}\rceil,      \]
then
\begin{eqnarray*}
(1-2\epsilon)\left [ 1-(k-1)\frac{\|L\|^2}{\epsilon n}\right ] &=&
(1-2\epsilon)\left [ 1-\frac{\epsilon n}{\|L\|^2}\frac{\|L\|^2}{\epsilon n}\right ]\\
&=&  (1-2\epsilon)0,
\end{eqnarray*}
and the process stops.
\end{corollary}

\end{document}